\documentclass[11pt]{article}
\usepackage[utf8]{inputenc}
\usepackage[english]{babel}
\usepackage{amsmath}
\usepackage{amsthm}
\usepackage{amssymb}
\usepackage{enumitem}
\usepackage{authblk}
\usepackage[all,cmtip]{xy}
\usepackage[symbol]{footmisc}

\makeatletter
\def\@fnsymbol#1{\ensuremath{\ifcase#1\or *\or **\or \ddagger\or \mathsection\or \mathparagraph\or \|\or \#\or \dagger\dagger\or \ddagger\ddagger \else\@ctrerr\fi}}

\newtheorem{theorem}{Theorem}[section]

\newtheorem{corollary}[theorem]{Corollary}
\newtheorem{lemma}[theorem]{Lemma}
\newtheorem{proposition}[theorem]{Proposition}

\theoremstyle{definition}
\newtheorem{definition}[theorem]{Definition}

\theoremstyle{remark} \theoremstyle{remark}
\newtheorem{remark}[theorem]{Remark}

\numberwithin{equation}{section}

\newcommand{\N}{\mathbb{N}}
\newcommand{\R}{\mathbb{R}}
\newcommand{\C}{\mathbb{C}}
\newcommand{\X}{\mathfrak{X}}
\newcommand{\f}{\varphi}

\begin{document}
	\title{\textbf{On Einstein hypersurfaces of a remarkable class of Sasakian manifolds}}
	\author[1]{Dario Di Pinto\thanks{{\em e-mail:} dario.dipinto@uniba.it}}
	\author[2]{Antonio Lotta\thanks{{\em e-mail:} antonio.lotta@uniba.it}}
	\affil[1,2]{Dipartimento di Matematica\\ Università degli Studi di Bari Aldo Moro\\ Via E. Orabona 4, 70125 Bari, Italy.}
	
	\maketitle
	
\begin{abstract}
	We present a non existence result of complete, Einstein hypersurfaces tangent to the Reeb vector field of a regular Sasakian manifold which fibers onto a complex Stein manifold. 
\end{abstract}
\medskip


\section{Introduction}

In \cite{Hasegawa}, I. Hasegawa established that a Sasakian space form with nonconstant sectional curvature 
admits no Einstein hypersurfaces.
The aim of this note is to prove a new non existence result concerning Einstein hypersurfaces of a relevant
class of {\em regular} Sasakian manifolds:

\bigskip	
{\bf Theorem.}
{\em	If $(M,\f,\xi,\eta,g)$ is a regular Sasakian manifold which fibers onto a complex Stein manifold, then $M$ does not admit any complete Einstein hypersurface tangent to $\xi$.}

\medskip
 We recall that a contact manifold $(M,\eta)$ is called
 regular provided the Reeb vector field $\xi$ of the contact form $\eta$ is, i.e. it determines a regular 1-dimensional foliation 
on $M$, so that the space $B=M/\xi$ of maximal integral curves
of $\xi$ is a manifold. When $M$ carries a  Sasakian  metric $g$ associated to $\eta$, yielding a Sasakian
structure $(\f,\xi,\eta,g)$ (we use the standard terminology and notation according to \cite{Blair}),
since $\mathcal{L}_\xi\varphi=0$ and $\mathcal{L}_\xi g=0$, $g$ induces in a natural way a metric $g'$ on $M/\xi$ and $\varphi$ also descends to an almost complex structure $J$. Denoting by $\pi:M\to B$ the canonical projection, it turns out by construction that 
$\pi$ is a Riemannian submersion with $\ker(d\pi)_x=\mathbb{R}\xi_x$ for every $x\in N$, and 
$$d\pi\circ\varphi=J\circ d\pi$$
and, morover $(B,J,g')$ is a K\"ahler manifold (see for instance \cite{Reckziegel} and \cite{Ogiue}).
Hence our assumption on the Sasakian manifold is that $B$, as a complex manifold, can be realized (up to a biholomorphism) as a closed
complex submanifold of some Euclidean space $\mathbb{C}^d$.

\medskip
For instance, according to a result due to H. Wu (\cite{Dillen-Verstraelen}, Theorem 4.9), it is known that every simply connected, complete K\"ahler manifold with non-positive sectional curvature is a Stein manifold;  in particular,  the Hermitian symmetric spaces of non-compact type are Stein manifolds, hence Takahashi's Sasakian globally $\varphi$-symmetric spaces of non-compact type (see \cite{Takahashi}) provide a wide class of examples of Sasakian manifolds to which our result applies.

\medskip
The proof of our result makes use of the natural $CR$ structure of $CR$ codimension 2 which is induced over any smooth hypersurface
$N\subset M$, under the assumption that $N$ is everywhere tangent to the Reeb vector field $\xi$
(see for instance \cite{Munteanu} and section \ref{prelim}). 
We establish a basic formula relating the Ricci tensor of $N$ and the trace of a distinguished scalar Levi form 
of this $CR$ structure (see \eqref{Riccieq}), implying that, in the Einstein case, $\pi(N)$ is a weakly pseudoconcave real hypersurface
of $B$. Hence a non compactness result by
D. Hill and M. Nacinovich for weakly pseudoconcave $CR$ submanifolds of Stein manifolds \cite{Hill-Nacinovich} is
invoked to get the conclusion.

\section{Preliminaries}
\label{prelim}
	Let's start by recalling the definitions of $CR$ manifolds, Levi-Tanaka forms and scalar Levi forms.
	In the following, given a vector bundle $E$ over a smooth differential manifold $M$, we will denote by $\Gamma(E)$ the $\mathcal{C}^\infty(M)$-module of global smooth sections of $E$.\\
	
	Let $M$ be a smooth real manifold of dimension $n$, and let $m,k\in\N$ such that $2m+k=n$. If $HM$ is a real vector subbundle of rank $2m$ of the tangent bundle $TM$ and $J:HM\to HM$ is a bundle isomorphism such that $J^2=-Id$, the couple $(HM,J)$ is called a \textit{CR structure} on $M$ if the following properties hold for all $X,Y\in\Gamma(HM)$:
	\begin{enumerate}[label=(\roman*)]
		\item $[JX,JY]-[X,Y]\in\Gamma(HM)$;
		\item $N_J(X,Y):=[JX,JY]-J[JX,Y]-J[X,JY]-[X,Y]=0$.
	\end{enumerate}
	In this case $(M,HM,J)$ is called a \textit{CR manifold of type (m,k)} and $m,k$ are the \textit{CR dimension} and the \textit{CR codimension} of the $CR$ structure, respectively.\\
	
	\begin{remark}\label{CRsubmanifold}
		Let $S$ be a real submanifold of a complex manifold $(M,J)$ and for any $p\in S$ set $$H_pS:=T_pS\cap J(T_pS).$$ 
		Because of the integrability of $J$, the couple $(HS,J_{|HS})$ canonically defines a $CR$ structure on $S$ if the dimension of $H_pS$ is constant. In this case $S$ is termed a \textit{CR submanifold} of $M$.\\ 
		In particular, this condition is always satisfied when $S$ is a real hypersurface of $M$ and hence $S$ is a $CR$ manifold of $CR$ codimension 1. 
	\end{remark}
\medskip

	\begin{definition}
		Let $(M,HM,J)$ be a $CR$ manifold of type $(m,k)$. Given a point $x\in M$, the \textit{Levi-Tanaka form of $M$ at $x$} is the bilinear map $$L_x:H_xM\times H_xM\to T_xM/H_xM$$ defined by 
		\begin{equation}\label{Levi-Tanaka}
		L_x(X,Y):=p_x([\tilde{X},J\tilde{Y}]_x)\quad \forall X,Y\in H_xM,
		\end{equation}
		where $\tilde{X},\tilde{Y}\in\Gamma(HM)$ are two arbitrary extensions of $X,Y$ and $p:TM\to TM/HM$ is the canonical projection on the quotient bundle $TM/HM$.
	\end{definition}
	It is known that $L_x$ is well defined, i.e. the value $p_x([\tilde{X},J\tilde{Y}]_x)$ only depends on the values of $\tilde{X},\tilde{Y}$ at $x$, that is on $X$ and $Y$.\\
	Moreover, according to (i) above, $L_x$ turns to be a vector valued symmetric Hermitian form on the holomorphic tangent space $H_xM$ with respect to the complex structure $J:=J_x$, that is
	\begin{equation}
	L_x(X,Y)=L_x(JX,JY),\quad L_x(X,Y)=L_x(Y,X)
	\end{equation}
	for all $X,Y\in H_xM$.
	
	\medskip
	Given a point $x$ on the $CR$ manifold $(M,HM,J)$, we will denote by $$H^0_xM:=\{\omega\in T^*_xM\ |\ \omega(X)=0\quad \forall X\in H_xM\}$$
	the annihilator of $H_xM\subset T_xM$. Then we recall the following definition.
	\begin{definition}
		Let $(M,HM,J)$ be a $CR$ manifold, $x\in M$ and $\omega\in H^0_xM$. The Hermitian form \begin{equation}
		\mathfrak{L}_\omega:H_xM\times H_xM\to\R\quad\text{s.t.}\quad \mathfrak{L}_\omega(X,Y):=\omega L_x(X,Y)
		\end{equation} 
		is called the \textit{scalar Levi form determined by $\omega$ at $x$}.
	\end{definition}
	
   The next lemma represents a sort of naturality property of the Levi-Tanaka form with respect a particular class of maps between $CR$ manifolds which preserve the $CR$ structures.
	\begin{definition}
		Let $(M,HM,J)$ and $(N,HN,J')$ be two $CR$ manifolds. A smooth map $\pi:M\to N$ is called \textit{CR map} if $d\pi(HM)\subset HN$ and $d\pi\circ J=J'\circ d\pi$.
	\end{definition}
	\begin{lemma}\label{naturality lemma}
		Let $(M,HM,J)$ and $(N,HN,J')$ be two $CR$ manifolds having the same $CR$ dimension, let $\pi:M\to N$ be a $CR$ map and assume that for every $x\in M$,  $(d\pi)_x:H_xM\to H_{\pi(x)}N$ is an isomorphism. Then, given $x\in M$, the following diagram commutes:
		$$\xymatrix{H_xM\times H_xM \ar[r]^{L_x}\ar[d]_{\pi_*\times\pi_*} &       T_xM/H_xM\ar[d]^{\pi_*}\\ H_yN\times H_yN \ar[r]_{L'_y} &T_yN/H_yN}$$
		where $y=\pi(x)$, $\pi_*=(d\pi)_x$ and $L_x,L'_y$ are the Levi-Tanaka forms of $M$ and $N$ respectively.
	\end{lemma}
\begin{proof}
	Let us denote by $p_x:T_xM\to T_xM/H_xM$ and $q_y:T_yN\to T_yN/H_yN$ the canonical projections. As an immediate consequence of the definition of $CR$ map, the differential $\pi_*$ descends to the quotient and, with abuse of notation, we still denote the quotient map by $\pi_*:T_xM/H_xM\to T_yN/H_yN$.\\
	Now consider $X,Y\in H_xM$: according to \eqref{Levi-Tanaka}, $L'_y(\pi_*X,\pi_*Y)=q_y[Z,J'W]_y$, where $Z,W\in \Gamma(HN)$ are two extensions of $\pi_*X$ and $\pi_*Y$. 
	Since for every $a\in M$, $(d\pi)_a:H_aM\to H_{\pi(a)}N$ is an isomorphism, we can define two extensions $\tilde{X},\tilde{Y}\in\Gamma(HM)$ of $X$ and $Y$ respectively by putting 
	$$\tilde{X}_a:=(d\pi)_a^{-1}(Z_{\pi(a)}),\quad 
	  \tilde{Y}_a:=(d\pi)_a^{-1}(W_{\pi(a)}).$$
	It turns out that $\tilde{X}$ and $\tilde{Y}$ are $\pi$-related to $Z$ and $W$ respectively, and hence $[\tilde{X},J\tilde{Y}]$ is $\pi$-related to $[Z,J'W]$ too, since $d\pi$ commutes with the almost complex structures $J$ and $J'$.
	Finally, we have:
	\begin{eqnarray*}
		\pi_*L_x(X,Y)&=&\pi_*(p_x[\tilde{X},J\tilde{Y}]_x)\\
					 &=&q_y(\pi_*[\tilde{X},J\tilde{Y}]_x)\\
					 &=&q_y[Z,J'W]_y\\
					 &=&L'_y(\pi_*X,\pi_*Y).
	\end{eqnarray*}
\end{proof}

\begin{corollary}\label{naturality corollary}
	In the same hypotesis and notation of the previous Lemma, for every $\psi\in H^0_yN$ one has that $\pi^*\mathfrak{L'}_\psi=\mathfrak{L}_{\pi^*\psi}$.
\end{corollary}
		
We remark that the scalar Levi forms $\mathfrak{L}_\omega$ are symmetric and hence it makes sense to consider their index $i(\mathfrak{L}_\omega)$, defined as the minimum between the number of positive and negative eigenvalues of $\mathfrak{L}_\omega$.\\
More specifically, we recall the following terminology from $CR$ geometry; see for instance \cite{Hill-Nacinovich-bis}.
	
	\begin{definition}
		Let $(M,HM,J)$ be a $CR$ manifold of type $(m,k)$ and let $x\in M$.\\
		$M$ is called \textit{pseudoconvex} at $x$ if $\mathfrak{L}_\omega$ is positive definite for some $\omega\in H^0_xM$.
		If there exists a global section $\omega\in\Gamma(H^0M)$ such that $\mathfrak{L}_\omega$ is positive definite at each point $x\in M$, $M$ is called \textit{strongly pseudoconvex}.\\
		$M$ is said \textit{pseudoconcave} at $x$ if $i(\mathfrak{L}_\omega)>0$ for every $\omega\in H_x^0M$, $\omega\neq0$.\\
		$M$ is said \textit{weakly pseudoconcave} at $x$ if $\mathfrak{L}_\omega=0$ or $i(\mathfrak{L}_\omega)>0$ for every $\omega\in H_x^0M$.
	\end{definition}

In this regard we recall that a Sasakian manifold $(M,\f,\xi,\eta,g)$, as defined in \cite{Blair}, is a particular kind of strongly pseudoconvex $CR$ manifold of hypersurface type, i.e. of $CR$ codimension 1. We shall refer to \cite{Blair} for the notation and basic facts concerning Sasakian geometry. We only remark that in this case the $CR$ structure is given by the contact distribution $\mathcal{D}=\ker\eta=\left<\xi\right>^\perp$ and the almost complex structure is $J=\f_{|\mathcal{D}}$. Therefore, for any $x\in M$, $H^0_xM$ is spanned by $\eta_x$ and, up to scaling, we have only one scalar Levi form $\mathfrak{L}_{\eta_x}$. Moreover, since $M$ is a contact metric manifold, the identity $$d\eta(X,Y)=g(X,\f Y)$$ yields that $$\mathfrak{L}_{\eta_x}=2{g_x}_{| H_xM\times H_xM}.$$
\medskip
	
We end this section by recalling the definition of Stein manifold (for more information, see for instance \cite{Dillen-Verstraelen}) and a theorem due to Hill and Nacinovich \cite{Hill-Nacinovich,Hill-Nacinovich-bis}, which provides a basic restriction to the topology of $CR$ weakly pseudoconcave submanifolds of a Stein manifold.
	
	\begin{definition}
		A \textit{Stein manifold} is a closed complex submanifold of $\C^d$, for some $d\ge1$.
	\end{definition}
	
	\begin{theorem}\label{Hill-Nacinovich Theorem}
		Every weakly pseudoconcave $CR$ submanifold of a Stein manifold cannot be compact. 
	\end{theorem}
	
\section{Main result}
Let $(M,\f,\xi,\eta,g)$ be a Sasakian manifold and let $N$ be a hypersurface of $M$, tangent to the Reeb vector field $\xi$. 
At each point $x\in N$, let us consider the linear subspace of $T_xN$ defined by
$$H_xN:=\{X\in T_xN\ |\ X\perp\xi_x\ \text{and}\ \f X\in T_xN\}.$$
Observe that, if $\nu\in T_xN^\perp$ is a unit normal vector at $x$,
then we have the following orthogonal decomposition:
$$T_xN=\left<\xi_x\right>\oplus\left<\f\nu\right>\oplus H_xN.$$
It follows that $HN$ is a subbundle of $TN$ with constant rank and 
in \cite{Munteanu} M. Munteanu proved that the couple $(HN,\f_{|HN})$ defines a $CR$ structure of $CR$ codimension 2 on $N$. 
We remark that he assumes the orientability of $N$, but this is unnecessary for our aim and the result holds true even if $N$ is not orientable.\\
The $CR$ structure $(HN,\f_{|HN})$ on the hypersurface $N$ allows us to consider, for every unit normal vector $\nu$, the scalar Levi form  $\mathfrak{L}_\omega$ attached to the covector 
\begin{equation}\label{omega}
	\omega(X)=g_x(X,\f\nu) \quad\forall X\in T_xN.
\end{equation}
We shall denote this scalar Levi form with the symbol $\mathfrak{L}_\nu$ and in the following proposition we establish the relationship between $\mathfrak{L}_\nu$ and the second fundamental form of the hypersurface $N$.
\begin{proposition}
	Let $(M,\f,\xi,\eta,g)$ be a Sasakian manifold and let $N\subset M$ be a hypersurface, tangent to $\xi$, with second fundamental form $\alpha$. Let $\nu$
	be a unit normal vector at some point $x\in N$. Then one has:
	\begin{equation}\label{Lnu-s.f.f}
	\mathfrak{L}_\nu(X,X)=g_x(\alpha(X,X)+\alpha(\f X,\f X),\nu)
	\end{equation}	
	for every $X\in H_xN$.
\end{proposition}
\begin{proof}
		First we recall that Sasakian manifolds are characterized by means of the following identity, involving the covariant derivatives of $\f$ with respect to the Levi-Civita connection (see \cite{Blair}):
	\begin{equation}\label{Sasaki}
	(\nabla_X\f)Y=g(X,Y)\xi-\eta(Y)X.
	\end{equation} 
	Now, fix $x\in N$, $X\in H_xN$ and consider a smooth section in $\Gamma(HN)$ which extends $X$
	and a local normal vector field extending $\nu$. Then $\varphi X$ is again tangent to $N$.
	Using the fact that $X$, $\varphi X$ and $\varphi\nu$ are all orthogonal to $\xi$ and identity (\ref{Sasaki}), we get:
	\begin{eqnarray*}
		&&\mathfrak{L}_\nu(X,X)=\\
		&=&g_x([X,\varphi X],\varphi\nu)=\\
		&=&g_x(\nabla_X\varphi X,\varphi\nu)-g_x(\nabla_{\varphi X}X,\varphi\nu)=\\
		&=&g_x(\varphi\nabla_X X,\varphi\nu)+g_x(\varphi\nabla_{\varphi X}X,\nu)=\\
		&=&g_x(\nabla_X X,\nu)+g_x(\nabla_{\varphi X}{\varphi X},\nu)=\\	
		&=&g_x(\alpha(X,X)+\alpha(\f X,\f X),\nu).
	\end{eqnarray*}
\end{proof}
We shall use this formula to establish an identity relating the trace (with respect to $g$) of $\mathfrak{L}_\nu$ and the Ricci tensor field of $N$.\\
Hereinafter we will denote with an overline the relevant geometric entities of the hypersurface $N$ (Levi-Civita connection, curvature, etc.).

\begin{proposition}\label{Ricci}
	Let $(M^{2n+1},\f,\xi,\eta,g)$ be a Sasakian manifold and let $N\subset M$ be a hypersurface tangent to $\xi$. Let $x\in N$
	and let $\nu\in T_xN^\perp$ be a unit normal vector. Then one has:

	\begin{equation}\label{Riccieq}
		\mathrm{\overline{Ric}}(\xi,\f\nu)=\frac12\mathrm{tr}(\mathfrak{L}_\nu).
	\end{equation}
\end{proposition}
\begin{proof}
	By a well known property of Sasakian manifolds (see \cite{Blair}), for every $X\in\X(M)$, 
	\begin{equation}\label{curvature}
		R(\xi,X)X=g(X,X)\xi-\eta(X)X.
	\end{equation}
 Then, given $X\in\Gamma(HN)$, since $\xi$ and $X$ are normal to $\f\nu$, it follows that
\begin{equation}
	R(\xi,X,\f\nu,X)=g(R(\xi,X)X,\f\nu)=0.
\end{equation}
	Since $\f X=-\nabla_X\xi$ is still tangent to $N$, we also deduce that the normal component of $\nabla_X\xi$ vanishes, i.e. $\alpha(X,\xi)=0$. Moreover,
	\begin{equation}
		\alpha(\xi,\f\nu)=g(\nabla_{\f\nu}\xi,\nu)\nu=g(-\f^2\nu,\nu)\nu=\nu.
	\end{equation}
    Therefore, by using the Gauss formula, for every $X\in\Gamma(HN)$ we have that
	\begin{equation}\label{N-curvature}
	   	\overline{R}(\xi,X,\f\nu,X)=
	  	 g(\alpha(X,X),\nu).
	\end{equation}
	Thus, fixed a local orthonormal frame of $TN$ of type $\{\xi,\f\nu,E_i,\f E_i\}_{i=1,\dots,n-1}$, with $E_i,\f E_i\in\Gamma(HN)$, from \eqref{N-curvature} and \eqref{Lnu-s.f.f} we get:
	\begin{eqnarray*}
		\mathrm{\overline{Ric}(\xi,\f\nu)}&=&\sum_{i=1}^{n-1}\left[\overline{R}(\xi,E_i,\f\nu,E_i)+\overline{R}(\xi,\f E_i,\f\nu,\f E_i)\right]\\
		&=&\sum_{i=1}^{n-1} g(\alpha(E_i,E_i)+\alpha(\f E_i,\f E_i),\nu)\\
		&=&\sum_{i=1}^{n-1} \mathfrak{L}_\nu(E_i,E_i)=\frac12\mathrm{tr}(\mathfrak{L}_\nu),
	\end{eqnarray*}
	where the last equality follows from the fact that $\mathfrak{L}_\nu$ is Hermitian and symmetric.
\end{proof}

\bigskip 	Now we come to the proof of our main result.

\smallskip

\begin{theorem}
	If $(M,\f,\xi,\eta,g)$ is a regular Sasakian manifold which fibers onto a complex Stein manifold, then $M$ does not admit any complete Einstein hypersurface tangent to $\xi$.
\end{theorem}
\begin{proof}
	Assume by contradiction that $M$ admits a complete Einstein hypersurface $N$ tangent to $\xi$, with Einstein constant $c$. \\
	Let $\overline{\nabla}$ be the Levi-Civita connection of $N$. Since $\nabla_\xi\xi=0$, from the Gauss equation we deduce that $\overline\nabla_\xi\xi=0$. Moreover, since $\xi$ is a Killing vector field on $N$, the operator $A_\xi:=-\overline\nabla\xi$ is skew-symmetric and hence $$\mathrm{\overline{Ric}}(\xi,\xi)=-\mathrm{div}(A_\xi\xi)-\mathrm{tr}(A_\xi^2)=-\mathrm{tr}(A_\xi^2)\ge0.$$
	It follows that $$c=cg(\xi,\xi)=\mathrm{\overline{Ric}}(\xi,\xi)\ge0.$$
	If $c=0$, then $A_\xi=0$, i.e. $\xi$ is $\overline{\nabla}$-parallel and this leads to a contradiction. Indeed, if we consider $X\in\Gamma(HN)$, with $X\neq0$, from $\overline{\nabla}_X\xi=0$ and the Gauss equation we would get
	$$-\f X=\nabla_X\xi=\alpha(X,\xi),$$
	where $-\f X\in\Gamma(HN)$ is non zero and tangent to $N$, while $\alpha(X,\xi)$ is normal.\\
	Therefore $c>0$ and, because of completeness of $N$, Myers' theorem ensures that $N$ is compact.\\
	Moreover, from Proposition \ref{Ricci} we have: 
	\begin{equation}\label{tr=0}
		\mathrm{tr}(\mathfrak{L}_\nu)=2\mathrm{\overline{Ric}(\xi,\f\nu)}=2cg(\xi,\f\nu)=0.
	\end{equation}
	Now, let $\pi:M\to M/\xi$ be the canonical projection, where $(M/\xi,J,g')$ is a Stein manifold; $\pi$ is a Riemannian submersion whose fibers are 1-dimensional submanifolds of $M$ tangent to $\xi$ and 
	\begin{equation}\label{fi/J}
		d\pi\circ\f=J\circ d\pi.
	\end{equation}
	Since at every $x\in M$, $\ker(d\pi)_x=\R\xi_x$, we have that $\pi_{|N}:N\to M/\xi$ has constant rank. Hence, according to Theorem 3.5.18 in \cite{Abraham-Marsden-Ratiu}, $S:=\pi(N)$ is a smooth hypersurface of $M/\xi$ and it carries a $CR$ structure (defined as in Remark \ref{CRsubmanifold}),  having the same $CR$ dimension of $N$. Moreover, \eqref{fi/J} implies that $\pi:N\to S$ is a $CR$ map, such that at every point $x\in N$ the differential $(d\pi)_x:H_xN\to H_{\pi(x)}S$ is an isomorphism.\\
	Fix a point $y=\pi(x)\in S$, with $x\in N$; if $\psi\in H_y^0S$, then $\pi^*\psi$ belongs to the vector space $H_x^0N$, which is spanned by $\omega$ and $\eta$, with $\omega$ as in \eqref{omega}. Actually, if $\pi^*\psi=\alpha\omega+\beta\eta$,  for some numbers $\alpha, \beta$, evaluating at $\xi$ we obtain $\beta=0$ and hence $\pi^*\psi=\alpha\omega$. Using Corollary \ref{naturality corollary} we get $$\pi^*\mathfrak{L'}_\psi=\mathfrak{L}_{\pi^*\psi}=\alpha\mathfrak{L}_\nu$$
    and by \eqref{tr=0} we conclude that $\mathrm{tr}(\mathfrak{L'}_\psi)=0$,  so that  $\mathfrak{L'}_\psi=0$ or $i(\mathfrak{L'}_\psi)>0$. Therefore $S$ is a compact weakly pseudoconcave $CR$ hypersurface of the complex Stein manifold $M/\xi$, thus contradicting Theorem \ref{Hill-Nacinovich Theorem}. 
\end{proof}

With just a small change in the previous proof, we also get the following result.

\begin{theorem}
	If $(M,\f,\xi,\eta,g)$ is a regular Sasakian manifold which fibers on a complex Stein manifold, then $M$ cannot admit any compact hypersurface $N$, tangent to $\xi$ and such that at any point of $N$ $\xi$ is an eigenvector of the Ricci operator $\overline{Q}$ of $N$.
\end{theorem}
\begin{proof}
	It suffices to note that if $\overline{Q}\xi=\alpha\xi$ along $N$, with $\alpha\in\mathcal{C}^\infty(N)$, then one has $$\mathrm{tr}(\mathfrak{L}_\nu)=2\overline{\mathrm{Ric}}(\xi,\f\nu)=2g(\overline{Q}\xi,\f\nu)=0.$$
	Hence the proof ends with the same argument of the previous one.
\end{proof}

\end{document}